\documentclass[11pt]{article}
 \usepackage{amsfonts, latexsym,  amsthm, amsmath,tikz}
\usepackage[all]{xy}
\numberwithin{equation}{section}

\title{}
\author{Ivan Horozov\\
(with an appendix by Matt Kerr)}
\date{}

\newcommand{\beq}{\begin{equation}}
\newcommand{\eeq}{\end{equation}}
\newcommand{\beqa}{\begin{eqnarray}}
\newcommand{\eeqa}{\end{eqnarray}}
\newcommand{\beaa}{\begin{eqnarray*}}
\newcommand{\ben}{\begin{eqnarray*}}
\newcommand{\eaa}{\end{eqnarray*}}
\newcommand{\een}{\end{eqnarray*}}

\newcommand \nc {\newcommand}
\newtheorem{theorem}{Theorem}[section]
\newtheorem{lemma}[theorem]{Lemma}
\newtheorem{proposition}[theorem]{Proposition}
\newtheorem{corollary}[theorem]{Corollary}
\newtheorem{definition}[theorem]{Definition}
\newtheorem{example}[theorem]{Example}
\newtheorem{remark}[theorem]{Remark}
\newtheorem{conjecture}[theorem]{Conjecture}
\newtheorem{question}[theorem]{Question}
\nc \bth[1] {\begin{theorem}\label{t#1} } \nc \ble[1]
{\begin{lemma}\label{l#1} } \nc \bpr[1]
{\begin{proposition}\label{p#1} } \nc \bco[1]
{\begin{corollary}\label{c#1} } \nc \bde[1]
{\begin{definition}\label{d#1}\rm } \nc \bex[1]
{\begin{example}\label{e#1}\rm } \nc \bre[1]
{\begin{remark}\label{r#1}\rm } \nc \bcon[1]
{\begin{conjecture}\label{con#1}\rm } \nc \bque[1]
{\begin{question}\label{que#1}\rm }
\nc {\eth} { \end{theorem} } \nc {\ele} { \end{lemma} } \nc
{\epr}{ \end{proposition} } \nc {\eco} { \end{corollary} } \nc
{\ede} {\end{definition} } \nc {\eex} { \end{example} } \nc {\ere}
{\end{remark} } \nc {\econ} { \end{conjecture} } \nc {\eque}
{\end{question} }
\nc \leref[1]{Lemma \ref{l#1}} \nc \prref[1]{Proposition
\ref{p#1}} \nc \coref[1]{Corollary \ref{c#1}} \nc
\deref[1]{Definition \ref{d#1}} \nc \exref[1]{Example \ref{e#1}}
\nc \reref[1]{Remark \ref{r#1}} \nc
\conref[1]{Conjecture\ref{con#1}}

\def\e{\epsilon}

\def \di {{\mathrm d}}

\def \P {{\mathbb P}}

\def \N {{\mathbb N}}
\def \Z {{\mathbb Z}}

\def \C {{\mathbb C}}

\def \ord { {\mathrm{ord}} }

\nc \Wr {Wr} \nc \GRN { \Gr^{(N)} }
\nc \GRA[1] { \Gr_A^{(#1)} }   
\nc \GRAN { \GRA{N} } \nc \GrA[1] { \Gr_A(#1) }\nc \GrAa {
\GrA{\alpha} }
\nc \GRB[1] { \Gr_B^{(#1)} }   
\nc \GRBN { \GRB{N} } \nc \GrB[1] { \Gr_B(#1) } \nc \GrBb {
\GrB{\beta} }
\nc \GRMB[1] { \Gr_{MB}^{(#1)} }   
\nc \GRMBN { \GRMB{N} } \nc \GrMB[1] { \Gr_{MB}(#1) } \nc \GrMBb {
\GrMB{\beta} }

\begin{document}

\title{{\LARGE\bf{
Reciprocity Laws on Algebraic Surfaces via Iterated Integrals} 
}}

\maketitle

\begin{abstract} This paper presents a proof of reciprocity laws for the Parshin symbol and for 
two new local symbols, defined here, which we call $4$-function local symbols. The reciprocity laws for the Parshin symbol are proven using a new method - via iterated integrals.   The usefulness of this method is shown by two facts - first, by establishing new local symbols -  the $4$-function local symbols and their reciprocity laws and, second, by providing refinements of the Parshin symbol in terms of bi-local symbols, each of which satisfies a reciprocity law. The $K$-theoretic variant of the first $4$-function local symbol is defined in the Appendix. It differs by a sign from the one defined via iterated integrals. Both the sign and the $K$-theoretic variant of the $4$-function local symbol satisfy reciprocity laws.
\end{abstract}

Key words: reciprocity laws, complex algebraic surfaces, iterated integrals

MSC2010: 14C30, 32J25, 55P35.

\tableofcontents
\setcounter{section}{-1}

\section{Introduction}
\label{Intro}

This paper is the second one in a series of papers on reciprocity laws on varieties via iterated integrals (after \cite{rec1}).
We construct and prove reciprocity laws for both classical and new symbols. 
Here, we present a new prove of the reciprocity laws for the Parshin symbol, in addition to well-known approaches such
as \cite{P1}, \cite{P2}, \cite{Ka}, \cite{Kh}, \cite{FV}, \cite{R}. 
Besides proofs of the Parshin reciprocity laws, the new method gives new symbols on algebraic surfaces and new reciprocity laws.

The present paper is a corrected and substantially improved version of
the preprint ``Refinement of the Parshin symbol for surfaces''
\cite{Refinement}.  In an email to the author \cite{D2}, Deligne pointed out that the refinements of the Parshin symbols were not independent of choices of local uniformizers.
After examining carefully
the origin of the refinement - namely, iterated integrals of differential forms over
membranes - we realized that the refinement becomes independent
of local uniformizers by introducing bi-local symbols.
A key property of the bi-local symbols is that they are almost  
the same as the tame symbol on a curve, however, they are defined over surfaces. 

We call these symbols bi-local, since they depend on two points $P$ and $Q$.
We fix two points $P$ and $Q$ on a curve $C$ on a
surface $X$ and we localize using two uniformizers:  one for
the curve $C$ on which the points $P$ and $Q$ lie, and one for the
point $P$. We represent the uniformaizers by rational functions. Then we evaluate a certain rational function at the points $P$ and $Q$ and take the ratio of the two values.
One can think of the point $P$ as the point of interest and of the point $Q$ as a base point. The points $P$ and $Q$ are in the union of the support of the divisors of the functions $f_1,\dots,f_4$, such that $P$ belongs to an intersection of two irreducible components of the divisors and $Q$ belongs to only one component, serving as a base point of loops on the curve $C$, where $Q$ belongs. 

We construct a refinement of the Parshin symbol in the sense that the latter is a product of bi-local symbols and all of them satisfy a reciprocity law. We also introduce 4-function local symbols, which have similar properties Ð they can  be factorized in simpler bi-local symbols that satisfy reciprocity laws. Moreover, such a presentation in terms of bi-local symbols provides proofs of the reciprocity laws for the local symbols.

Three of the six symbols that compose the Parshin symbol for a surface have only values $\pm1$. The composition of the remaining three symbols, which gives the Parshin symbol up to a sign, will be used in a follow-up paper constructing a two-dimensional analogue of the Contou-Carr\'ere symbol and its reciprocity laws.

Another reason for using bi-local symbols is that they are computationally effective. 

Unlike the paper \cite{rec1}, where we used iterated
integrals over paths for reciprocity laws, here we define a higher
dimensional analogue, which we call {\it{iterated integrals over
membranes}}. It took five or six years to complete the many
details around these new ideas. An apology is due from the author
to the mathematical community and to my former student for
that delay.

The idea for {\it{iterated integrals over membranes}} had its genesis in
an attempt to generalize Manin's non-commutative modular symbol
\cite{M} to a non-commutative Hilbert modular symbol
\cite{ModSymb}, which remains an ongoing project. However,
I received an encouraging email from Manin \cite{M2} about my work \cite{ModSymb}. 

Before exploring reciprocity laws on surfaces, one has to establish reciprocity laws on curves, for example, the Weil reciprocity law. In \cite{rec1}, the proof of the Weil reciprocity is via iterated integrals. It uses only double iteration, in contrast to higher order iteration, which is used in the formulas for a parallel transport with respect to a connection (see \cite{rec1}.) Similarly, the Parshin symbol on a surface and the two $4$-function local symbols use relatively simple iterated integrals over membranes. More complicated iterated integrals over membranes might also be considered for the purpose of reciprocity laws. However, in general, they will produce very complicated formulas. Simpler formulas occur only when we consider at most double iteration. For double iteration, that produces the Parshin symbol for surfaces and the two $4$-function local symbols.

The sources of new symbols in our approach are iterated integrals. More precisely, every iterated integral leads to a reciprocity law. 
In \cite{rec1} Theorems 2.9 and  3.3, we use higher order iteration on a complex curve. Then the reciprocity laws are complicated. One can do the same for surfaces. However, we have chosen to consider at most double iterated integrals, which lead to relatively simple reciprocity laws. Over a surface there are three such (iterated) integrals: 

(i) a $2$-form leading to an analogue of ``the sum of the residues is zero";

(ii) iteration a $2$-form with a $1$-form - leading to the Parshin symbol; 

(iii) an iteration of a $2$-form with a $2$-form which leads to both $4$-function local symbols.

The algebraic varieties in this paper are defined over the complex numbers $\C$. However, all the constructions on a variety $X_K$ would work 
over any algebraically closed subfield $K\subset \C$, 
simply by considering the induced embedding $X_K\subset X_\C$.

There are several interesting formulas that we would like to introduce to the attention of the reader. 
For explaining the formulas defining  the reciprocity laws, it would be instructive to make a comparison with the Weil reciprocity law stated in terms of the tame symbol.
 
The divisor of non-zero rational functions $f$  on a complex smooth projective curve 
is formal sum 
\[(f)=\sum_i a_i P_i,\]
such that $P_i$'s  are points where $f$ has zeros or poles and
the coefficients $a_i\in \Z$
are the orders of vanishing of $f$ at the points $P_i$. Let also
\[(g)=\sum_j b_j Q_j.\]
Following Weil, we define
\[f((g))=\prod_jf(Q_j)^{b_j}\,\,\,\text{ and }\,\,\,g((f))=\prod_i g(P_i)^{a_i}.\]
\begin{theorem} (Weil reciprocity law)
If the support of the divisor of $f$ and the support of the divisor of $g$ are disjoint then
\[f((g))=g((f)).\]
\end{theorem}

Weil reciprocity law can be expressed it terms of the tame symbol, in order to include the cases when the support of $f$ and $g$ have common points.
The tame symbol on a curve $C$ is defined as
\[\{f,g\}_P=(-1)^{ab}\left(\frac{f^b}{g^a}\right)(P),\]
where $a=\ord_P(f)$ and $b=\ord_P(g)$.
If $Q$ is in the support of $g$ but not in the support of $f$ then
\[\{f,g\}_Q=f(Q)^b,\]
where $b=\ord_Q(g)$.
As a consequence 
\[\prod_{Q\in Support(g)}\{f,g\}_Q=f((g)).\]

We can express the Weil reciprocity, using the tame symbol.
\begin{theorem} (Weil reciprocity law in terms of the tame symbol)
The tame symbol satisfies the following reciprocity law
\[\prod_P\{f,g\}_P=1,\]
where the product is taken over all points $P$ of the smooth projective curve $C$.
\end{theorem}
Note that the tame symbol is (possibly) different from $1$ only when the point $P$ is in the union of the support of the divisors of $f$ and $g$.

Now let $X$ be a smooth complex  projective surface, let $C$ be a smooth curve on the surface $X$ and let $P$ be a point on the curve $C$.
For a non-zero rational function $f_k$ on the surface $X$, let 
\[a_k=\ord_C(f_k)\]
be the order of vanishing of $f_k$ on the curve $C$.
Let also $x$ be a rational function on the surface $X$, representing an uniformizer at the curve $C$, 
such that no pair of irreducible components of the support of the divisor of $x$ intersect at the point $P$.
Let
\[b_k=\ord_P((x^{-a_k}f_k)|_C).\]
We recall the Parshin symbol
\[\{f_1,f_2,f_3\}_{C,P}=(-1)^K\left(f_1^{D_1}f_2^{D_2}f_3^{D_3}\right)(P),\]
where
\[D_1=
\left|
\begin{tabular}{ll}
$a_2$&$a_3$\\
$b_2$&$b_3$
\end{tabular}
\right|,
\mbox{ }
D_2=
\left|
\begin{tabular}{ll}
$a_3$&$a_1$\\
$b_3$&$b_1$
\end{tabular}
\right|,
\mbox{ }
D_3=
\left|
\begin{tabular}{ll}
$a_1$&$a_2$\\
$b_1$&$b_2$
\end{tabular}
\right|
\]
and
\[K=a_1a_2b_3+a_2a_3b_1+a_3a_1b_2+b_1b_2a_3+b_2b_3a_1+b_3b_1a_2.\]
The Parshin symbol satisfies the following reciprocity laws.

\begin{theorem} Let $f_1,f_2,f_3$ be non-zero rational functions on a smooth (complex) projective surface $X$, the following reciprocity laws hold:

(a) \[\prod_P\{f_1,f_2,f_3\}_{C,P}=1,\] where the product is taken over all points $P$ over a fixed curve $C$. Here we assume that the union of the support of the divisors $\bigcup_{i=1}^3|div(f_i)|$ in $X$ have normal crossing and no three components have a common point.

(b) \[\prod_C\{f_1,f_2,f_3\}_{C,P}=1,\] where the product is taken over all curves $C$ passing through a fixed point $P$. Here we assume that the union of the support of the divisors 
$\bigcup_{i=1}^3|div(f_i)|$ in $\tilde{X}$ have normal crossings 
and no two components have a common point  with the exceptional curve $E$ in $\tilde{X}$ 
above the point $P$. 
We denote by $\tilde{X}$ the blow-up of $X$ at the point $P$.
\end{theorem}

We obtain both the Weil reciprocity law and the Parshin reciprocity law via iterated integrals.
Using the techniques of iterated integrals over membranes (Subsection 1.4), we define two new $4$-function local symbols.

\begin{definition} With the above notation, we define two new $4$-function local symbols:
$$\{f_1,f_2,f_3,f_4\}^{(1)}_{C,P}
=
(-1)^L
\frac{\left(\frac{f_1^{a_2}}{f_2^{a_1}}\right)^{a_3b_4-b_3a_4}}
{\left(\frac{f_3^{a_4}}{f_4^{a_3}}\right)^{a_1b_2-b_1a_2}}(P).$$
and
$$\{f_1,f_2,f_3,f_4\}^{(2)}_{C,P}
=
(-1)^L
\frac{\left(\frac{f_1^{a_2+b_2}}{f_2^{a_1+b_1}}\right)^{-(a_3b_4-b_3a_4)}}
{\left(\frac{f_3^{a_4+b_4}}{f_4^{a_3+b_3}}\right)^{-(a_1b_2-b_1a_2)}}(P),$$
where $L=(a_1b_2-b_1a_2)(a_3b_4-b_3a_4).$
\end{definition}
For them, we have the following reciprocity laws.
\begin{theorem} (Reciprocity laws for the new $4$-function local symbols)
Let $f_1,f_2,f_3$ be non-zero rational functions on a smooth (complex) projective surface $X$, 
the following reciprocity laws hold:

(a) \[\prod_P\{f_1,f_2,f_3,f_4\}^{(1)}_{C,P}=1,\] where the product is taken over all point $P$ of a fixed curve $C$. Here we assume that the union of the support of the divisors $\bigcup_{i=1}^4|div(f_i)|$ in $X$ have normal crossing and no three components have a common point.

(b) \[\prod_C\{f_1,f_2,f_3,f_4\}^{(2)}_{C,P}=1,\] where the product is taken over all curves $C$ passing through a fixed point $P$. 
Here we assume that the union of the support of the divisors 
$\bigcup_{i=1}^4|div(f_i)|$ in $\tilde{X}$ have normal crossings 
and no two components have a common point  with the exceptional curve $E$ in $\tilde{X}$ 
above the point $P$. 
We denote by $\tilde{X}$ the blow-up of $X$ at the point $P$.
\end{theorem}

Our approach is based on new types of symbols which we call Òbi-local symbolsÓ. They allow us to refine the local symbols that we study (the Parshin symbol, the 4-function symbols) in the sense that the local symbols of interest are presented as products of the bi-local symbols and then reciprocity laws are proven for the latter.

For the reader interested in $K$-theoretic approach, we have included a second proof of the  reciprocity laws for the new $4$-function local symbols, based on  Milnor $K$-theory. It can be found in Section 4 and the Appendix. 

We learned from Pablos Romo that recently a third proof of the reciprocity laws for the $4$-function local symbols, 
as well as new results about refinements of the Parshin symbol were obtained \cite{PR2}.

Let us relate the work in this paper to other results in this area. Brylinski and McLaughlin (see 
\cite{BrMcL}) used gerbes to define the Parshin symbol. Here we give an alternative, more analytic 
approach, based on iterated integrals over membranes. We should mention a few other approaches  
to tame symbols and to the Parshin symbol, for example, \cite{D1}, \cite{OZh}, \cite{R}.

\subsection*{Structure of the paper}
In Subsection 1.1, we recall basic properties of iterated integrals
over paths. Then, in Subsection 1.2, we prove Weil reciprocity by establishing first a
reciprocity law for a bi-local symbol, and then removing the
dependence on the base point, we recover the Weil reciprocity for the tame symbol on a curve. Subsection 1.3 gives a construction of two foliations. They are needed for the definition of iterated integrals on membranes, presented in Subsection 1.4.  Such integrals are the key technical ingredient in this paper.

Section 2, contains the first type of reciprocity laws for the Parshin symbol and for the first $4$-function local symbol, where the product of the symbols
is over all points $P$ of a fixed curve $C$ on a surface $X$. The proofs are based on the reciprocity laws for bi-local symbols expressed as iterated integrals on membranes. Certain products of bi-local symbols become local symbols such as the Pashin symbol or the first $4$-function local symbol.  We call such products a refinement of the Pashin symbol or a refinement of the first $4$-function local symbol.

Section 3 is about the second type of reciprocity laws, where the
product of the symbols is taken over all curves $C$ on $X$ passing
though a fixed point $P$.  We first establish a technical result
about the Parshin symbol and first $4$-function symbol under
blow-up.  We also define bi-local symbols suitable for the
second type of reciprocity law.  Then the corresponding reciprocity laws
are proven for the bi-local symbols, the Parshin symbol,
and the second $4$-function local symbol. The bi-local symbols in Section 3 provide a second type of refinement of the Parshin symbol and for the second $4$-function local symbol.

For convenience of the reader, we conclude with Section 4, by giving an alternative proof of the reciprocity laws of the $4$-function local symbols based on Milnor $K$-theory.

\subsection*{Acknowledgments}
I am indebted to Parshin, Osipov and the entire Algebra Seminar at
the Steklov Institute in Moscow for their interest, numerous
questions and discussions in relation to the results in this
paper. I would also like to thank Manin and Goncharov for their
encouragement to develop my ideas about iterated integrals on
membranes, and Deligne for his valuable remarks on an earlier
version of this paper. I would like to thank M. Kerr for the
useful discussions, for proofs in Subsection 1.3 and for the
careful reading of the paper. I would also like to thank Pablos Romo for his interest in this manuscript.

Finally, I gratefully acknowledge the financial support and good
working environment provided by Max Planck Institute for Mathematics, Brandeis University, University of
Tuebingen and Washington University in St. Louis during the period
of working on this and related papers.

\section{Geometric and analytic background}
\subsection{Iterated path integrals on complex curves}
This Subsection contains a definition and properties of iterated integrals, which will be used for the definition of bi-local symbols and for another proof of the Weil reciprocity law in Subsection 1.2. 

Let $C$ be a smooth complex curve. Let $f_1$ and $f_2$ be two non-zero rational functions on $C$.
Let
$$\gamma:[0,1]\rightarrow C$$
be a path, which is a continuous, piecewise
differentiable function on the unit interval.

\begin{definition}
We define the following iterated integral
$$\int_\gamma \frac{df_1}{f_1}\circ\frac{df_2}{f_2}
=
\int_{0<t_1<t_2<1}\gamma^*\left(\frac{df_1}{f_1}\right)(t_1)\wedge\gamma^*\left(\frac{df_2}{f_2}\right)(t_2).$$
\end{definition}

The two Lemmas below are due to K.-T. Chen \cite{Ch}.
\begin{lemma}
\label{lem h-invariant} An iterated integral over a path $\gamma$
on a smooth curve $C$ is homotopy invariant with respect to a
homotopy, fixing the end points of the path $\gamma$.
\end{lemma}

\begin{lemma}
\label{lem composition} If $\gamma=\gamma_1\gamma_2$ is a
composition of two paths, where the end of the first path $\gamma_1$ is the
beginning of the second path $\gamma_2$, then
$$\int_{\gamma_1\gamma_2} \frac{df_1}{f_1}\circ\frac{df_2}{f_2}
=
\int_{\gamma_1} \frac{df_1}{f_1}\circ\frac{df_2}{f_2}
+
\int_{\gamma_1} \frac{df_1}{f_1}\int_{\gamma_2}\frac{df_2}{f_2}
+
\int_{\gamma_2} \frac{df_1}{f_1}\circ\frac{df_2}{f_2}.$$
\end{lemma}

Let $\sigma$ be a simple loop around a point $P$ on $C$ with a
base point $Q$. Let $\sigma=\gamma\sigma_0\gamma^{-1}$, where
$\sigma_0$ is a small loop around $P$, with a base the point $R$ and let $\gamma$ be a path
joining the  points $Q$ with $R$. 

The following Lemma is essential for the proof of the Weil reciprocity (see also  \cite{rec1}).
\begin{lemma} With the above notation, the following holds
\label{lemma large loop}
$$\int_\sigma \frac{df_1}{f_1}\circ\frac{df_2}{f_2}
=
\int_\gamma \frac{df_1}{f_1}\int_{\sigma_0}\frac{df_2}{f_2}
+\int_{\sigma_0} \frac{df_1}{f_1}\circ\frac{df_2}{f_2}
+\int_{\sigma_0}\frac{df_1}{f_1}\int_{\gamma^{-1}} \frac{df_2}{f_2}.$$
\end{lemma}
\begin{proof} First, we use Lemma \ref{lem composition} for the
composition $\gamma\sigma_0\gamma^{-1}$. 
We obtain

\begin{equation}
\label{eq large loop}
\begin{tabular}{lll}
\\
$\int_\sigma \frac{df_1}{f_1}\circ\frac{df_2}{f_2}$
&=
$\int_\gamma \frac{df_1}{f_1}\circ\frac{df_2}{f_2}
+\int_\gamma \frac{df_1}{f_1}\int_{\sigma_0}\frac{df_2}{f_2}
+\int_{\sigma_0}\frac{df_1}{f_1}\circ\frac{df_2}{f_2}$\\
\\
&$+\int_{\gamma}\frac{df_1}{f_1}\int_{\gamma^{-1}} \frac{df_2}{f_2}
+\int_{\sigma_0}\frac{df_1}{f_1}\int_{\gamma^{-1}} \frac{df_2}{f_2}
+\int_{\gamma^{-1}} \frac{df_1}{f_1}\circ\frac{df_2}{f_2}$\\
\end{tabular}
\end{equation}
\vspace{.3cm}

Then, we use the homotopy
invariance of iterated integrals, Lemma \ref{lem h-invariant}, for
the path $\gamma\gamma^{-1}$. Thus,
$$0=\int_{\gamma\gamma^{-1}}\frac{df_1}{f_1}\circ\frac{df_2}{f_2}.$$
Finally, we use Lemma \ref{lem composition} for the composition of
paths $\gamma\gamma^{-1}$. That gives
\begin{equation}
\label{eq simplification}
0=\int_{\gamma\gamma^{-1}}\frac{df_1}{f_2}\circ\frac{df_2}{f_2}
=\int_\gamma\frac{df_1}{f_1}\circ\frac{df_2}{f_2}
+\int_\gamma\frac{df_1}{f_1}\int_{\gamma^{-1}}\frac{df_2}{f_2}
+\int_{\gamma^{-1}}\frac{df_1}{f_1}\circ\frac{df_2}{f_2}.
\end{equation}
The Lemma \ref{lemma large loop} follows from Equations \eqref{eq large loop} and \eqref{eq simplification}. 
\end{proof}
\subsection{Weil reciprocity via iterated path integrals}
Here, we present a proof of the Weil reciprocity law, based on iterated integrals and bi-local symbols. 
This method will be generalized in the later Subsections in order to prove reciprocity laws on complex surfaces. Similar ideas about the Weil reciprocity law are contained in \cite{rec1}, however, without bi-local symbols.

Let $x$ be a rational function on a curve $C$, representing an uniformizer at $P$.
Let
$$a_i=\ord_P(f_i).$$ and let
$$g_i=x^{-a_i}f_i.$$
Then
$$\frac{df_i}{f_i}=a_i\frac{dx}{x}+\frac{dg_i}{g_i}.$$
Let $\sigma_0^\e$ be a small loop around the point $P$, whose
points are at most at distance $\e$ from the point $P$. One can take the metric inherited from the Fubini-Study metric on $\P^k$. Put $\sigma_0^\e=\sigma_0$ in Lemma \ref{lemma large loop}, then
$$\int_{\gamma}\frac{df_1}{f_1}\int_{\sigma_0^\e}\frac{df_2}{f_2}
=
2\pi i a_2\int_{\gamma}\frac{df_1}{f_1}
=
2\pi i a_2\left(a_1\int_\gamma\frac{dx}{x}+\int_\gamma\frac{dg_1}{g_1}\right).$$
Similarly,
$$\int_{\sigma_0^\e}\frac{df_1}{f_1}\int_{\gamma^{-1}}\frac{df_2}{f_2}
=
2\pi i a_1\left(-a_2\int_\gamma\frac{dx}{x}-\int_\gamma\frac{dg_2}{g_2}\right).$$
From \cite{rec1}, we have that
\begin{equation}
\label{eq limit}
\lim_{\e\rightarrow 0}\int_{\sigma_0^\e}\frac{df_1}{f_1}\circ\frac{df_2}{f_2}
=
\frac{(2\pi i)^2}{2}a_1a_2.
\end{equation}
Using Lemma \ref{lemma large loop}, we obtain
$$\int_\sigma\frac{df_1}{f_2}\circ\frac{df_2}{f_2}
=
2\pi i\left(a_2\log(g_1)-a_1\log(g_2)+\pi ia_1a_2\right)|_Q^P.$$
After exponentiation, we obtain
\begin{lemma} With the above notation the following holds
\label{le tame}
$$\exp\left(\frac{1}{2\pi i}\int_{\sigma}\frac{df_1}{f_2}\circ\frac{df_2}{f_2}\right)
=
(-1)^{a_1a_2}\frac{g_1^{a_2}}{g_2^{a_1}}(P)\left(\frac{g_1^{a_2}}{g_2^{a_1}}(Q)\right)^{-1}
=
(-1)^{a_1a_2}\frac{f_1^{a_2}}{f_2^{a_1}}(P)\left(\frac{f_1^{a_2}}{f_2^{a_1}}(Q)\right)^{-1}
$$
\end{lemma}
\begin{definition}(Bi-local symbol on a curve)
\label{def bi-local}
With the above notation, we define a bi-local symbol
\begin{equation}
\label{eq biloc}
\{f_1,f_2\}_P^Q=
(-1)^{a_1a_2}\frac{f_1^{a_2}}{f_2^{a_1}}(P)\left(\frac{f_1^{a_2}}{f_2^{a_1}}(Q)\right)^{-1}.
\end{equation}
\end{definition}

Let the curve $C$ be of genus $g$ 
and let $P_1,\dots,P_n$ be the points of the union of the support of
the divisors of $f_1$ and $f_2$. Let $\sigma_1,\dots,\sigma_n$ be
simple loops around the points $P_1,\dots,P_n$, respectively. Let
also $\alpha_1,\beta_1,\dots,\alpha_g,\beta_g$ be the $2g$ loops
on the curve $C$ such that
$$\pi_1(C,Q)=<\sigma_1,\dots,\sigma_n,\alpha_1,\beta_1,\dots,\alpha_n,\beta_n>/\sim,$$
where $\delta\sim 1$, for
$$\delta=\prod_{i=1}^n\sigma_i\prod_{j=1}^g[\alpha_j,\beta_j].$$
From Theorem 3.1 in \cite{rec1}, we have
\begin{lemma}
\label{le commutator}
 $$\int_{\alpha\beta\alpha^{-1}\beta^{-1}}\frac{df_1}{f_1}\circ\frac{df_2}{f_2}
=
\int_\alpha\frac{df_1}{f_1}\int_\beta\frac{df_2}{f_2}-\int_\alpha\frac{df_2}{f_2}\int_\beta\frac{df_1}{f_1}.$$
\end{lemma}

Using the above result, we obtain that
$$0=\int_\delta \frac{df_1}{f_1}\circ\frac{df_2}{f_2}
\in
(2\pi i)^2\Z+\sum_{i=1}^n\int_{\sigma_i} \frac{df_1}{f_1}\cdot\frac{df_2}{f_2},$$
where the sum is over simple loops $\sigma_i$ around each of the points $P_i$.
Then we obtain:
\begin{theorem}
(Reciprocity law for the bi-local symbol \eqref{eq biloc})
With the above notation, the following holds
$$\prod_P\{f_1,f_2\}_P^Q=1.$$
\end{theorem}
If we want to make the above reciprocity law into a reciprocity law for a local symbol we have to
remove the dependency on the base point $Q$. This can be achieved
in the following way: In the reciprocity law for the bi-local
symbol, the dependency on $Q$ is
$$\prod_P f_1(Q)^{a_2}f_2(Q)^{-a_1} = f_1(Q)^{(2\pi i)^{-1}\sum_P Res_P\frac{df_2}{f_2}}=$$
$$f_2(Q)^{-(2\pi i)^{-1}\sum_P Res_P\frac{df_1}{f_1}}=f_1(Q)^0f_2(Q)^0=1.$$

Thus, we recover Weil reciprocity:
\begin{theorem} 
\label{thm Weil rec}
(Weil reciprocity)
The local symbol
\[\{f_1,f_2\}_P
=
(-1)^{a_1a_2}\frac{f_1^{a_2}}{f_2^{a_1}}(P).
\] 
satisfies the following reciprocity law
\[\prod_P\{f_1,f_2\}_P=1,\]
where the product is over all points $P$ in $C$.
\end{theorem}

\subsection{Two foliations on a surface}
\label{subsec 2 foliations}
The goal of this Subsection is to construct two foliations on a
complex projective algebraic surface $X$ in $\P^k$.
This is an algebraic-geometric material, 
needed for the definition of iterated integrals on membranes, presented in Subsection 1.4.

Let $f_1$, $f_2$, $f_3$
and $f_4$ be four non-zero rational functions on the surface $X$.
Let
$$C\cup C_1\cup\dots\cup C_n=\bigcup_{i=1}^4|div(f_i)|,$$
where we fix one of the irreducible components $C$.
Let $$\{P_1,\dots,P_N\}=C\cap (C_1\cup\dots\cup C_n).$$ We can
assume that the curves $C,C_1,\dots,C_n$ are smooth and that the
intersections are transversal (normal crossings) and no three of them
intersect at a point, by allowing blow-ups on the surface $X$.

The two foliations have to satisfy the following
\\

\noindent
{\bf{Conditions:}}
\begin{enumerate}
\item There exists a foliation $F'_v$  such that
  \begin{enumerate}
    \item $F'_v=(f-v)_0$ are the level sets of a rational function $$f:X\rightarrow \P^1,$$ for small values of $v$, (that is, for $|v|<\e$ for a chosen $\e$);
    \item $F'_v$ is smooth for all but finitely many values of $v$;
    \item $F'_v$ has only nodal singularities;
    \item $ord_C(f)=1$;
    \item $R_i\notin C_j,$ for $i=1,\dots,M$ and $j=1,\dots,n$, where $$\{R_1,\dots,R_M\}=C\cap (D_1\cup\dots\cup D_m)$$ and $$F'_0=(f)_0=C\cup D_1\cup\dots\cup D_m.$$
  \end{enumerate}
\item There exists a foliation $G_w$ such that
  \begin{enumerate}
    \item $G_w=(g-w)_0$ are the level sets of a rational function $$g:X\rightarrow \P^1;$$
    \item $G_w$ is smooth for all but finitely many values of $w$;
    \item $G_w$ has only nodal singularities;
    \item $g|_C$ is non constant.
\end{enumerate}
\item Coherence between the two foliations $F'$ and $G$:
  \begin{enumerate}
    \item All but finitely many leaves of the foliation $G$ are transversal to the curve $C$.
    \item $G_{g(P_i)}$ intersects the curve $C$ transversally, for $i=1,\dots,N$. (For definition of the points $P_i$ see the beginning of this Subsection.)
    \item $G_{g(R_i)}$ intersects the curve $C$ transversally, for $i=1,\dots,M$. (For definition of the points $R_i$
see condition 1(e).)
\end{enumerate}
\end{enumerate}

The existence of $f\in\C(X)^\times$ satisfying properties 1(a-d) is a direct consequence from the following result,
which follows immediately from (a special case of) a result of Thomas (\cite{Th}, Theorem 4.2).
\begin{theorem}
\label{thm foliations} Consider a smooth curve $C$ in a smooth
projective surface $X$, with hyperplane section $H_X$. There exists a large constant $N\in\N$ and
a pencil in $|NH_X|$, given as the level sets $(f-x)_0$ of some
rational function $f$ such that $(f-x)_0$ is smooth for all but
finitely many values of $x$, at which it has only nodal
singularities, and  $C\subset (f)_0$.
\end{theorem}

Moreover, a general choice of $g\in \C(X)^\times$
will satisfy 2(a-d) and 3(a-c). (For instance, the quotient of two
generic linear forms on $\P^k$ restricted to $C$ will not have
branch points in $\{P_i\}\cup\{R_j\}$.)


It remains to examine property 1(e). The proof of Theorem 4.2 in [op. cit.] contains the basic \\

\noindent {\bf{Observation:}} {\it{The base locus of the linear system $H^0(I_C(N))$ is the smooth curve $C$ for$N>>0$. So by Bertini's theorem the general element of the linear
system is smooth away from $C$.}} \\

Consider $C\subset X$. By the Observation, there exists ${\cal{F}}\in
H^0(X,{\cal{O}}(N))$ such that $\ord_C({\cal{F}})=1$ and
$({\cal{F}})=C+D$, where $D$ is a second smooth curve on $X$,
meeting $C$ transversally (if at all). \\

\noindent {\bf{Claim:}} {\it{We may choose ${\cal{F}}$  so that
condition 1(e) holds, that is, $R_i\notin C_j$ for each $i,j$,
where $\{R_1,\dots,R_M\}=C\cap D$. Equivalently, $C\cap D\cap
C_j=\emptyset.$ }}

\begin{proof} Define $H^0(I_C(N))^{reg}$ to be the subset of
$H^0(X,I_C(N))$ whose elements ${\cal{F}}$ satisfy
$\ord_C({\cal{F}})=1$ and $({\cal{F}})=C+D$ as above. Assume that
for every $N>>0$ and ${\cal{F}}\in H^0(I_C(N))^{reg}$ we have $D\cap C\cap
C_j\neq\emptyset$ for some particular $j$. If we obtain a contradiction (for some $N$) then
the claim is proved, since this is a closed condition for each $j$.

According to our assumption, $({\cal{F}})$ always has an ordinary
double point at the intersection $\Delta:=C\cap C_j \neq \emptyset$. In the exact sequence
$$0\rightarrow H^0(X,I^2_C(N))\rightarrow H^0(X,I_C(N))\rightarrow H^0(C,{\cal{N}}^{*}_{C/X}(N))
\rightarrow H^1(X,I^2_C(N)),$$
the last term vanishes by
(\cite{GH}, Vanishing Theorem B) for $N$ sufficiently large.
Hence, every section over $C$  of  the twisted conormal sheaf
${\cal{N}}^*_{C/X}(N)$ has a zero along $\Delta=C\cap C_j$.

Next consider the exact sequence
$$0\rightarrow H^0(C,I_\Delta\otimes{\cal{N}}^*_{C/X}(N))\rightarrow H^0(C,{\cal{N}}^*_{C/X}(N))\rightarrow \C^{|\Delta|}\rightarrow H^1(C,I_\Delta\otimes {\cal{N}}^*_{C/X}(N)).
$$
The last term vanishes again by [loc. cit.]. Denote
the third arrow by $ev_\Delta$. Then we can take a section of
${\cal{N}}^*_{C/X}(N)$ not vanishing on $\Delta$ simply by taking an element
in the preimage of $ev_\Delta(1,\dots,1)$. This produces the
desired contradiction.
\end{proof}

Consider a metric on the surface
$X$, which respects the complex structure. For example, we can
take the  metric inherited from the Fubini-Study metric on $\P^k$
via the embedding  $X\rightarrow \P^k$. Let $U^\e_1,\dots, U^\e_M$
be disks of radii $\e$ on $C$, centered respectively at
$R_1,\dots,R_M$. 
Let \[C_0=C-\bigcup_{j=1}^M U^{\e}_j-\{P_1,\dots,P_N\}.\]
\begin{definition} 
\label{def F}
With the above notation,
let $F_v$ be the connected component of
\[F'_v-\left(\bigcup_{i=1}^M G_{g(U^\e_i)}\right)\cap F'_v,\]
containing $C_0$, for $|v|<<\e$, where
\[G_{g(U^\e_i)}=\bigcup_{w\in U^\e_i}G_{g(w)}\]
\end{definition}

\begin{lemma}
\label{lemma F_0}
With the above notation, for small values of $|v|$, we have that each leaf $F_v$ 
is a continuous deformation of $F_0=C_0$, preserving homotopy type.
\end{lemma}
\proof From Property 3(c), it follows that $C$ and $D_i$ meet at
$R_j$ (if at all) at a non-zero angle. At the intersection $R_i$,
locally we can represent the curves by $xy=0$. The deformation
leads to $v-xy=0$, which is a leaf of $F'$, locally near $R_i$.
Consider a disk $U$ of radius $\e_i$ at $(x,y)=(0,0)$ in the
$xy$-plane. Then for $|v|<<\e_i$ we have that $U$ separates $F'_v$
into $2$ components, one close to the $x$-axis and the other close
to the $y$-axis. We do the same for each of the points
$R_1,\dots,R_M$ and we take the minimum of the bounds $\e_i$. 
Then $F_v$ will consist of points close to the curve $C_0$.
\qed

\subsection{Iterated integrals on a membrane. Definitions and properties}
\label{subsec membrane}
In this Subsection, we define types of iterated integrals over membranes, needed in most of this manuscript. 

Let $\tau$ be a simple loop around $C_0$ in $X-C_0-\left(\bigcup_{i=1}^M G_{g(U^\e_i)}\right)$.
Let $\sigma$ be a loop on the curve $C_0$.
We define a {\it{membrane}} $m_{\sigma}$  associated to a loop $\sigma$ in $C^0$ by
$$m_{\sigma}:[0,1]^2\rightarrow X$$
and 
$$m_{\sigma}(s,t)\in F_{f(\tau(t))}\cap G_{g(\sigma(s))}.$$
Note that for fixed values of $s$ and $t$, we have that
$$F_{f(\tau(t))}\cap G_{g(\sigma(s))}$$
consists of finitely many points, where $F$ and $G$ are foliations
satisfying the Conditions in Subsection \ref{subsec 2 foliations} and Lemma \ref{lemma F_0}.

{\bf{Claim:}} The image of $m_\sigma$ is a torus.

Indeed, consider a tubular neighborhood around a loop $\sigma$ on the curve $C_0$. One can take the following tubular neighborhood: 
\[\bigcup_{|v|<\e}F_v\cap G_{g(\sigma)}\]
of $F_v\cap G_{g(\sigma)}$. Its boundary is $F_{f(\tau)}\cap G_{g(\sigma)}$, where $\tau$ is a simple loop around $C_0$ on $X-\bigcup_{i=1}^n C_i -\bigcup_{j=1}^m D_j$ and $|f(\tau(t))|=\e$.

We shall define the simplest type of iterated integrals over membranes. Also, we are going to construct local symbols in terms of iterated integrals $I_1,I_2,I_3,I_4$ on membranes, defined below. 
 
We define the following differential forms
$$A(s,t)=m^*\left(\frac{\di f_1}{f_1}\wedge\frac{\di f_2}{f_2}\right)(s,t)$$
$$b(s,t)=m^*\left(\frac{\di f_3}{f_3}\right)(s,t)$$
and
$$B(s,t)=m^*\left(\frac{\di f_3}{f_3}\wedge\frac{\di f_4}{f_4}\right)(s,t).$$

The first diagram
\begin{center}
\begin{tikzpicture}
\draw[step=1cm] (0,0) grid (1,1);
\draw (0.5,-0.5)node{$s$};
\draw (-0.5,0.5)node{$t$};
\draw (0.5,0.5)node{$A$};
\end{tikzpicture}
\end{center}
denotes 
$$I_1=\int_0^1\int_0^1A(s,t).$$
The second diagram
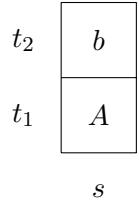
\begin{center}
\begin{tikzpicture}
\draw[step=1cm] (0,0) grid (1,2);
\draw (0.5,-0.5)node{$s$};
\draw (-0.5,0.5)node{$t_1$};
\draw (-.5,1.5)node{$t_2$};
\draw (0.5,0.5)node{$A$};
\draw(0.5,1.5)node{$b$};
\end{tikzpicture}
\end{center}
denotes
$$I_2=\int_0^1\int\int_{0<t_1<t_2<1}A(s,t_1)\wedge b(s,t_2).$$
The third diagram
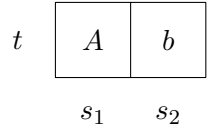
\begin{center}
\begin{tikzpicture}
\draw[step=1cm] (0,0) grid (2,1);
\draw (0.5,-0.5)node{$s_1$};
\draw (1.5,-.5)node{$s_2$};
\draw (-0.5,0.5)node{$t$};
\draw (0.5,0.5)node{$A$};
\draw(1.5,.5)node{$b$};
\end{tikzpicture}
\end{center}
denotes
$$I_3=\int\int_{0<s_1<s_2<1}\int_0^1A(s_1,t)\wedge b(s_2,t).$$
And the fourth diagram
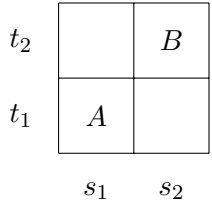
\begin{center}
\begin{tikzpicture}
\draw[step=1cm] (0,0) grid (2,2);
\draw (0.5,-0.5)node{$s_1$};
\draw (1.5,-.5)node{$s_2$};
\draw (-0.5,0.5)node{$t_1$};
\draw (-.5,1.5)node{$t_2$};
\draw (0.5,0.5)node{$A$};
\draw(1.5,1.5)node{$B$};
\end{tikzpicture}
\end{center}
denotes
$$I_4=\int\int_{0<s_1<s_2<1}\int\int_{0<t_1<t_2<1}A(s_1,t_1)\wedge B(s_2,t_2).$$

Local symbols will be defined via the above four types of iterated integrals. The integrals that we define below, used for defining bi-local symbols, are a technical tool for proving reciprocity laws for the local symbols. Bi-local symbols also satisfy reciprocity laws.
 
Consider the dependence of $\log(f_i(m(s,t)))$ on the variables $s$
and $t$ via the parametrization of the membrane $m$. 
\begin{definition}
Let $$l_i(s,t)=\log(f_i(m(s,t)))$$
\end{definition}
We have
$$\di l_i(s,t)=
\frac{\partial l_i(s,t)}{\partial s}\di s+\frac{\partial l_i(s,t)}{\partial t}\di t.
$$
$$b(s,t)=\di l_3(s,t)$$
\begin{equation}A(s,t)
=
\frac{\partial l_1(s,t)}{\partial s}\frac{\partial l_2(s,t)}{\partial t}\di s\wedge\di t 
- 
\frac{\partial l_1(s,t)}{\partial t}\frac{\partial l_2(s,t)}{\partial s}\di s\wedge\di t 
\end{equation}

\begin{equation}B(s,t)
=
\frac{\partial l_3(s,t)}{\partial s}\frac{\partial l_4(s,t)}{\partial t}\di s\wedge\di t 
- 
\frac{\partial l_3(s,t)}{\partial t}\frac{\partial l_4(s,t)}{\partial s}\di s\wedge\di t 
\end{equation}

The above equations express the differential forms $A$, $B$ and $b$ is terms of monomials in terms of  first derivatives of $l_1,l_2,l_3,l_4$. 
We are going to define bi-local symbols associated to monomials in 
 first derivatives of $l_1,l_2,l_3,l_4$, which occur in 
 \[A(s,t), \text{ }A(s,t_1)\wedge b(s,t_2),\text{ }A(s_1,t)\wedge b(s_2,t),\text{ and }A(s_1,t_2)\wedge B(s_2,t_2)\]

\begin{definition}(Iterated integrals on membranes)
\label{def I} Let $f_1,\dots, f_{k+l}$ be rational functions on
$X$, where the pairs $(k,l)$ will be superscripts of the integrals. Let $m$ be
a membrane as above. We define:\\
\\
\noindent (a) $I^{(1,1)}(m;f_1,f_2)=$\\
\begin{eqnarray*}
=\int_0^1\int_0^1
\left(\frac{\partial l_1(s,t)}{\partial s}\di s\right)\wedge\left(\frac{\partial l_2(s,t)}{\partial t}\di  t\right)
\end{eqnarray*}

\noindent (b) $I^{(1,2)}(m;f_1,f_2,f_3)=$\\
\begin{eqnarray*}
=\int\int\int_{0\leq s\leq 1;\text{ }0\leq t_1\leq t_2\leq1}
\left(\frac{\partial l_1(s,t_1)}{\partial s}\frac{\partial l_2(s,t_1)}{\partial t_1}\di s\wedge\di t_1\right) \wedge
\left(\frac{\partial l_3(s,t_2)}{\partial t_2}\di t_2\right)
\end{eqnarray*}

\noindent (c) $I^{(2,1)}(m;f_1,f_2,f_3)=$\\
\begin{eqnarray*}=\int\int\int_{0\leq s_1\leq s_2 \leq 1;\text{ }0\leq t \leq1}
\left(\frac{\partial l_1(s_1,t)}{\partial s_1}\frac{\partial l_2(s_1,t)}{\partial t}\di s_1\wedge\di t\right)\wedge
\left(\frac{\partial l_3(s_2,t)}{\partial s_2}\di s_2\right)
\end{eqnarray*}

\noindent (d) $I^{(2,2)}(m;f_1,f_2,f_3,f_4)=$\\
\begin{eqnarray*}=\int\int\int\int_{0\leq s_1\leq s_2\leq 1;\text{ }0\leq t_1\leq t_2\leq1}
\left(\frac{\partial l_1(s_1,t_1)}{\partial s_1}\frac{\partial l_2(s_1,t_1)}{\partial t_1}
\di s_1\wedge\di t_1\right)
\wedge\\
\wedge
\left(\frac{\partial l_3(s_2,t_2)}{\partial s_2}\frac{\partial l_4(s_2,t_2)}{\partial t_2}
\di s_2\wedge\di t_2\right)
\end{eqnarray*}
\end{definition}

\begin{proposition}
\label{prop membranes}
(a) $I_1= I^{(1,1)}(m;f_1,f_2)-I^{(1,1)}(m;f_2,f_1)$;

(b) $I_2=I^{(1,2)}(m;f_1,f_2,f_3)-I^{(1,2)}(m;f_2,f_1,f_3)$;

(c) $I_3=I^{(2,1)}(m;f_1,f_2,f_3)-I^{(2,1)}(m;f_2,f_1,f_3)$;

(d) $I_4=I^{(2,2)}(m;f_1,f_2,f_3,f_4)-I^{(2,2)}(m;f_2,f_1,f_3,f_4)-I^{(2,2)}(m;f_1,f_2,f_4,f_3)+I^{(2,2)}(m;f_2,f_1,f_4,f_3)$;

\end{proposition}

Consider a metric on the projective surface $X$ inherited from the Fubini-Study metric on $\P^k$.
Let $\tau$ be a simple loop
around the curve $C$ of distance at most $\e$ from $C$. We are
going to take the limit as $\e\rightarrow 0$. Informally, the
radius of the loop $\tau$ goes to zero. Then we have the following
lemma.

\begin{lemma}
\label{le limit}
With the above notation the following holds:

(a)
\begin{equation*}
\lim_{\e\rightarrow 0}I^{(1,1)}(m_{\sigma},f_1,f_2)
=
(2\pi i)Res \frac{df_2}{f_2}\int_{\sigma}\frac{df_1}{f_1}
\end{equation*}

(b)
\begin{equation*}
\lim_{\e\rightarrow 0}I^{(1,2)}(m_\sigma,f_1,f_2,f_3)
=
\frac{(2\pi i)^2}{2}Res\frac{df_2}{f_2}Res\frac{df_3}{f_3}\int_\sigma\frac{df_1}{f_1}
\end{equation*}

(c)
\begin{equation*}
\lim_{\e\rightarrow 0}I^{(2,1)}(m_\sigma,f_1,f_2,f_3)
=
-(2\pi i)Res\frac{df_2}{f_2}\int_\sigma\frac{df_1}{f_1}\circ\frac{df_3}{f_3}
\end{equation*}

(d)
\begin{equation*}
\lim_{\e\rightarrow 0}I^{(2,2)}(m_\sigma,f_1,f_2,f_3,f_4)
=
-\frac{(2\pi i)^2}{2}Res\frac{df_2}{f_2}Res\frac{df_4}{f_4}
\int_{\sigma}\frac{df_1}{f_1}\circ\frac{df_3}{f_3}
\end{equation*}
\end{lemma}
\proof First, we consider the integrals in parts (a) and (c), where there is 
integration with respect to the variable $t$ in the definition of the
membrane $m$. Let $m(s,\cdot)$ denote the loop obtained by fixing the first variable $s$ and varying the second variable $t$.Then, there is no iteration along the loop
$m(s,\cdot)$  around the curve $C$, for fixed value of $s$. 
Using Properties 1(d) and e(b), 
the integration over the loop  $m(s,\cdot)$ gives us a single residue. 
This process is independent of the base point of the loop $m(s,\cdot)$.
That proves parts (a) and (c).

For parts (b) and (d), we have a double iteration along the loop $m(s,\cdot)$ around the curve $C$,
where the value of $s$ is fixed and the second argument varies.
After taking the limit as $\e$ goes to $0$, the integral along $m(s,\cdot),$
with respect to $t_1$ and $t_2$, becomes a product of two residues (see Equation \eqref{eq limit}),
which are independent of a base point.
That proves parts (b) and (d).
\qed

\section{First type of reciprocity laws}

\subsection{Reciprocity laws for bi-local symbols}
In this Subsection, we define bi-local symbols and prove their reciprocity laws. 
Using them, in the following two Sections, we establish the first type of reciprocity laws for the Parshin symbol and for a new $4$-function new symbol. By a first type of reciprocity law, we mean that the product of the local symbols is taken over all points $P$ of a fixed curve $C$ on the surface $X$.

Consider the fundamental group of $C_0$. We recall that $C_0$ is essentially the curve $C$ without several intersection points and without several open neighborhoods. More precisely, 
$$C_0=C-\left(\bigcup_{j=1}^m G_{U^\e_j}\right)\cap C-\left(\bigcup_{i=1}^n C_i\right)\cap C.$$
where $U^\e_j$ is a small neighborhood of $R_j$ on the complex curve $C$.
We recall the notation for the intersection points
$$\{P_1,\dots,P_N\}=C\cap \left(C_1\cup\dots\cup C_n\right),$$
$$\{R_1,\dots,R_M\}=C\cap \left(D_1\cup\dots\cup D_m\right),$$
Let
$$\pi_1(C_0,Q)=<\sigma_1,\dots,\sigma_n,\alpha_1,\beta_1,\dots,\alpha_g,\beta_g>/\sim$$
be a presentation of the fundamental group, 
where
$$\delta\sim 1,$$
for
$$\delta=\prod_{i=1}^n\sigma_i\prod_{j=1}^g[\alpha_j,\beta_j].$$

We are going to drop the indices  $i$ and $j$. Thus, we are going to write
$P$ instead of $P_i$ or $R_j$ and $\sigma$ instead of $\sigma_i$.
Consider the definition of a membrane $m_\sigma$, associated to a loop $\sigma$, given in the beginning
of Subsection \ref{subsec membrane}. Let $m_\sigma(s,\cdot)$ be the loop obtained by
fixing the variable $s$ and letting the second argument vary.
Similarly, $m_\sigma(\cdot,t)$ denotes the loop obtained by fixing
the variable $t$ and letting the first argument vary. 
\begin{definition}
Let $a_k=\ord_C (f_k)$ and $b_k=\ord_P((x^{-a_k}f_k)|_C),$ where $x$ is a rational function, representing an uniformizer such that $\ord_C(x)=1$ and $P$ is not an intersection of any two of the components of the divisor of $x$.
\end{definition}
It is straightforward to represent the order of vanishing as residues, given by the following:  
\begin{lemma}
\label{lemma ab} We have
\[
a_k=\frac{1}{2\pi i}\int_{m_\sigma(s,\cdot)}\frac{df_k}{f_k}
\,\,\,\,\,\,\,\mbox{ and }\,\,\,\,\,\,\,
b_k=\frac{1}{2\pi i}\int_{m_\sigma(\cdot,t)}\frac{df_k}{f_k}.
\]
\end{lemma}

Using properties 1(d) and 3(b), we should think of $m_\sigma(\cdot,t)$ and $m_\sigma(s,\cdot)$ as
translates of $\sigma$ and of $\tau$, respectively. Then the above
integrals are residues, which detect the order of vanishing. For
example $a_k$ is the order of vanishing of $f_k$ along a generic
point of $C$. Then the following theorem holds, whose proof is
immediate from Lemmas  \ref{le limit} and \ref{lemma ab}.

\begin{theorem}
\label{thm int-sym}
(a)
$$
(2\pi i)^{-2}
\lim_{\e\rightarrow 0}
I^{(1,1)}(m_\sigma,f_1,f_2)
=a_2b_1
,$$

(b)
$$
(2\pi i)^{-2}
\lim_{\e\rightarrow 0}
I^{(1,2)}(m_\sigma,f_1,f_2,f_3)
=
(\pi i) a_2a_3b_1,$$

(c) $$\exp\left(
(2\pi i)^{-2}
\lim_{\e\rightarrow 0}
I^{(2,1)}(m_\sigma,f_1,f_2,f_3)
\right)
=
\left(\{f_2,f_3\}_P^Q\right)^{-a_1}
,$$

(d)
$$
\exp\left(
\frac{2}{(2\pi i)^{3}}
\lim_{\e\rightarrow 0}
I^{(2,2)}(m_\sigma,f_1,f_2,f_3,f_4)
\right)
=
\left(\{f_1,f_3\}_P^Q\right)^{-a_2a_4}
.$$
\end{theorem}

Let us denote by $\alpha$ the loop $\alpha_j$ and by $\beta$ the loop $\beta_j$.
Then the following lemma holds

\begin{lemma}

(a)
$$
(2\pi i)^{-2}
\lim_{\e\rightarrow 0}
I^{(1,1)}(m_{[\alpha,\beta]}f_1,f_2)
=0
,$$

(b)
$$
(2\pi i)^{-2}
\lim_{\e\rightarrow 0}I^{(1,2)}(m_{[\alpha,\beta]},f_1,f_2,f_3)
=
0
,$$

(c) $$
\exp\left(
(2\pi i)^{-2}
\lim_{\e\rightarrow 0}I^{(2,1)}(m_{[\alpha,\beta]},f_1,f_2,f_3)
\right)
=
1
,$$

(d)
$$
\exp\left(
\frac{2}{(2\pi i)^{3}}
\lim_{\e\rightarrow 0}I^{(2,2)}(m_{[\alpha,\beta]},f_1,f_2,f_3,f_4)
\right)
=
1
.$$
\end{lemma}
\proof It follows from Lemmas \ref{le limit} and \ref{le
commutator}. A more modern proof follows from the well-definedness
of the integral Beilinson regulator on $K_2$ on the level of
homology (see \cite{Ke}.)

\begin{definition}(Bi-local symbols on a surface)
For a simple loop $\sigma$ around a point $P$ in $C_0$, based at $Q$, let
\begin{equation*}
\label{eq log symbol}
Log^{(i,j)}[f_1,\dots,f_{i+j}]^{(1),Q}_{C,P}=\lim_{e\rightarrow 0} I^{i,j}(m_\sigma,f_1,\dots,f_{i+j}),
\end{equation*}

\begin{equation*}
\label{eq  (1,2)-symbol}
^{1,2}[f_1,f_2,f_3]^{(1),Q}_{C,P}
=
\exp
\left(
(2\pi i)^{-2}
\lim_{\e\rightarrow 0}
I^{(1,2)}(m_\sigma,f_1,f_2,f_3)
\right),
\end{equation*}
\end{definition}

\begin{equation*}
\label{eq  (2,1)-symbol}
^{2,1}[f_1,f_2,f_3]^{(1),Q}_{C,P}
=
\exp
\left(
(2\pi i)^{-2}
\lim_{\e\rightarrow 0}
I^{(2,1)}(m_\sigma,f_1,f_2,f_3)
\right),
\end{equation*}

\begin{equation*}
\label{eq  (2,2)-symbol}
^{2,2}[f_1,f_2,f_3,f_4]^{(1),Q}_{C,P}
=
\exp
\left(
\frac{2}{(2\pi i)^{3}}
\lim_{\e\rightarrow 0}
I^{(2,2)}(m_\sigma,f_1,f_2,f_3,f_4)
\right).
\end{equation*}

The following reciprocity laws hold for the above bi-local symbols.
\begin{theorem}
\label{thm rec bilocal}
(a) $\sum_{P}Log^{1,1}[f_1,f_2]^{(1),Q}_{C,P}=0$.

(b) ${\prod_P} ^{1,2}[f_1,f_2,f_3]^{(1),Q}_{C,P}=1.$

(c) ${\prod_P} ^{2,1}[f_1,f_2,f_3]^{(1),Q}_{C,P}=1.$

(d) ${\prod_P} ^{2,2}[f_1,f_2,f_3,f_4]^{(1),Q}_{C,P}=1.$
\end{theorem}
\begin{proof} Parts (b), (c) and (d) follow directly from Theorem \ref{thm
int-sym} and from Weil reciprocity. Part (a) follows again from
Theorem \ref{thm int-sym} and the theorem that the sum of the
residues of a differential form on a curve is zero.  \end{proof}

\subsection{Parshin symbol and its first reciprocity law.}
In this Subsection, we construct a refinement of the Parshin symbol in terms of six bi-local symbols.
Using this presentation of the Parshin symbol, Definition \ref{def refinement} and Theorem \ref{thm refinement},
we prove the first reciprocity of the Parshin symbol (Theorem \ref{thm parshin1}).

\begin{definition}
\label{def refinement}
We define the following bi-local symbol
\begin{eqnarray*}
Pr_{C,P}^Q=
\left({^{1,2}[f_1,f_2,f_3]_{C,P}^{(1),Q}}\right)
\left({^{1,2}[f_2,f_3,f_1]_{C,P}^{(1),Q}}\right)
\left({^{1,2}[f_3,f_1,f_2]_{C,P}^{(1),Q}}\right)
\times\\
\times
\left({^{2,1}[f_1,f_2,f_3]_{C,P}^{(1),Q}}\right)
\left({^{2,1}[f_2,f_3,f_1]_{C,P}^{(1),Q}}\right)
\left({^{2,1}[f_3,f_1,f_2]_{C,P}^{(1),Q}}\right)
\end{eqnarray*}
at the points $P=P_i\in C\cap (C_1\cup\dots\cup C_n)$ and a fixed point $Q$ in $C-C\cap (C_1\cup\dots\cup C_n)$.
\end{definition}

Using Theorem \ref{thm int-sym} parts (b) and (c), we obtain:

\begin{theorem}
\label{thm refinement}
 (Refinement of the Parshin symbol)
We have the following explicit  formula
$$Pr_{C,P}^Q=(-1)^K\frac{\left(f_1^{D_1}f_2^{D_2}f_3^{D_3}\right)(P)}
{\left(f_1^{D_1}f_2^{D_2}f_3^{D_3}\right)(Q)},$$
where
$$D_1=
\left|
\begin{tabular}{ll}
$a_2$&$a_3$\\
$b_2$&$b_3$
\end{tabular}
\right|,
\mbox{ }
D_2=
\left|
\begin{tabular}{ll}
$a_3$&$a_1$\\
$b_3$&$b_1$
\end{tabular}
\right|,
\mbox{ }
D_3=
\left|
\begin{tabular}{ll}
$a_1$&$a_2$\\
$b_1$&$b_2$
\end{tabular}
\right|
$$
and
$$K=a_1a_2b_3+a_2a_3b_1+a_3a_1b_2+b_1b_2a_3+b_2b_3a_1+b_3b_1a_2.$$
\end{theorem}

Note that $Pr_{C,P}^Q$ is essentially the Parshin symbol, which can be defined in the following way
\begin{definition} (The Parshin symbol)
$$\{f_1,f_2,f_3\}_{C,P}=(-1)^K\left(f_1^{D_1}f_2^{D_2}f_3^{D_3}\right)(P).$$
\end{definition}

The only difference between the two symbols is the constant factor in $Pr_{C,P}^Q$, depending only on the base point $Q$ (the denominator of $Pr_{C,P}^Q$). Rescaling by that constant leads to the Parshin symbol.

\begin{theorem} 
\label{thm parshin1}
(First reciprocity law for the Parshin symbol)
For the Parshin symbol, the following reciprocity law holds
$$\prod_P\{f_1,f_2,f_3\}_{C,P}=1,$$
where the product is taken over points $P$  in $C\cap (C_1\cup\dots\cup C_n)$. (When $P$ is another point of $C$ then the symbol is trivial.) Here we assume that the union of the support of the divisors $\bigcup_{i=1}^3|div(f_i)|$ in $X$ have normal crossing and no three components have a common point.

\end{theorem}
\proof We are going to use the reciprocity laws for bi-local symbols
stated in Theorem \ref{thm rec bilocal} parts (b) and (c). Then the
reciprocity law for the bi-local symbol $Pr_P^Q$ follows. 
There is relation between the Parshin symbol  and $Pr_P^Q$, namely,
$$\{f_1,f_2,f_3\}_{C,P}=Pr_{C,P}^Q\left(f_1^{D_1}f_2^{D_2}f_3^{D_3}\right)(Q).$$
Now, we remove the dependence on the base point $Q$. In order to do that, note that
$$\prod_P f_1(Q)^{D_1}=g_1(Q)^{\sum_P D_1}.$$
Here $g_1=x^{-a_1}f_1$, where $x$ is a rational function on the surface $X$, representing an uniformazer at the curve $C$, such that the components of the divisor of $x$ do not intersect at the points $P$ or $Q$.
Moreover,
$$D_1=(2\pi i)^{-2}\left(Log^{1,1}[f_2,f_3]_P^{(1),Q}-Log^{1,1}[f_3,f_2]_P^{(1),Q}\right)$$
by Theorem \ref{thm int-sym} part (a) and Proposition \ref{prop membranes} part (a).
Using Theorem \ref{thm rec bilocal} part (a), for the above equality, we obtain
$$\sum_P D_1=0.$$
Therefore,
$$\prod_P g_1(Q)^{D_1}=1.$$
Similarly,
$$\prod_P g_2(Q)^{D_2}=1\mbox{ and }\prod_P g_3(Q)^{D_3}=1,$$
where $g_k=x^{-a_k}f_k.$

\subsection{New $4$-function local symbol and its first reciprocity law}
In this Subsection, we define a new $4$-function local symbol on a surface. We also express the new 
$4$-function local symbol as a product of bi-local symbols (Definition \ref{def 4f bilocal} and 
Proposition \ref{prop 4f refinement}), which serves as a refinement similar to the refinement of the Parshin 
symbol in Subsection 2.2. Using the reciprocity laws for bi-local symbols established in 
Subsection 2.1, we obtain the first type of reciprocity law for the new $4$-function local symbol 
(Theorem \ref{thm rec1 4fn}). 

\begin{definition}
\label{def 4f bilocal}
We define the following bi-local symbol, which will lead to the $4$-function local symbol on a surface.
\begin{eqnarray*}
PR_{C,P}^Q
=
\left(^{2,2}[f_1,f_2,f_3,f_4]_P^{(1),Q}\right)
\left(^{2,2}[f_1,f_2,f_4,f_3]_P^{(1),Q}\right)^{-1}
\times\\
\times
\left(^{2,2}[f_2,f_1,f_3,f_4]_P^{(1),Q}\right)^{-1}
\left(^{2,2}[f_2,f_1,f_4,f_3]_P^{(1),Q}\right).
\end{eqnarray*}
\end{definition}
Using Theorem \ref{thm int-sym}, part (d), we obtain:
\begin{proposition}
\label{prop 4f refinement}
Explicitly, the bi-local symbol $PR_{C,P}^Q$ is given by
\begin{equation}
\label{eq 4bilocal}
PR_{C,P}^Q=
(-1)^L
\frac{\left(\frac{f_1^{a_2}}{f_2^{a_1}}\right)^{a_3b_4-b_3a_4}}
{\left(\frac{f_3^{a_4}}{f_4^{a_3}}\right)^{a_1b_2-b_1a_2}}(P)
\cdot
\left(
\frac{\left(\frac{f_1^{a_2}}{f_2^{a_1}}\right)^{a_3b_4-b_3a_4}}
{\left(\frac{f_3^{a_4}}{f_4^{a_3}}\right)^{a_1b_2-b_1a_2}}(Q)
\right)^{-1},
\end{equation}
where
$$L=
(a_1b_2-a_2b_1)(a_3b_4-a_4b_3).$$
\end{proposition}

\begin{definition}($4$-function local symbol)
\label{def 4-function} With the above notation, we define a $4$-function local symbol
$$\{f_1,f_2,f_3,f_4\}^{(1)}_{C,P}
=
(-1)^L
\frac{\left(\frac{f_1^{a_2}}{f_2^{a_1}}\right)^{a_3b_4-b_3a_4}}
{\left(\frac{f_3^{a_4}}{f_4^{a_3}}\right)^{a_1b_2-b_1a_2}}(P).$$
\end{definition}
It is an easy exercise to check that the symbol
$\{f_1,f_2,f_3,f_4\}^{(1)}_{C,P}$ is independent of the choices of local
uniformizers.
See also the Appendix for $K$-theoretical approach for the $4$-function local symbol.
Note that the relation between the bi-local symbol $PR_{C,P}^C$ and the local symbol 
$\{f_1,f_2,f_3,f_4\}^{(1)}_{C,P}$ is only a constant factor depending on the base point $Q$. There is a similar relation between the bi-local symbol $Pr_{C,P}^Q$ and the Parshin symbol 
$\{f_1,f_2,f_3\}_{C,P}$.

\begin{theorem}
\label{thm rec1 4fn}
(Reciprocity law for the $4$-function local symbol)
 The following reciprocity law for the $4$-function local symbol on a surface holds
$$\prod_P \{f_1,f_2,f_3,f_4\}^{(1)}_{C,P}=1,$$
where the product is taken over points $P$ on a fixed curve $C$. Here we assume that the union of the support of the divisors $\bigcup_{i=1}^4|div(f_i)|$ in $X$ have normal crossing and no three components have a common point.
\end{theorem}
\begin{proof} Using Theorem \ref{thm rec bilocal} part (d), we obtain that the bi-local
symbol $PR_{C,P}^Q$ satisfies a reciprocity law, namely,
\begin{equation}
\prod_P PR_{C,P}^Q=1,
\end{equation}
where the product is over all points $P$ in $C\cap (C_1\cup\dots\cup C_n)$.
In order to complete
the proof of Theorem \ref{thm rec1 4fn}, we proceed similarly to the proof of
the first Parshin reciprocity law. Namely,
\begin{equation}
\label{eq 4f}
\prod_P g_1(Q)^{a_2(a_3b_4-a_4b_3)}=g_1(Q)^{a_2\sum_P a_3b_4-a_4b_3}=g(Q)^{b_2\cdot 0}=1,
\end{equation}
where $g_1=x^{-a_1}f_1$ and
$x$ is a rational function representing an uniforminzer at the curve $C$, such that the components of the divisor of $x$ do not intersect at the points $P$ or $Q$.
The last equality of \eqref{eq 4f} holds, because
$$a_3b_4-a_4b_3
=
(2\pi i)^{-2}\left(Log^{1,1}[f_3,f_4]_{C,P}^{(1),Q}-Log^{1,1}[f_4,f_3]_{C,P}^{(1),Q}\right)=0$$
and
$$\sum_P (2\pi i)^{-2}\left(Log^{1,1}[f_3,f_4]_{C,P}^{(1),Q}-Log^{1,1}[f_4,f_3]_{C,P}^{(1),Q}\right)=0,$$
by Theorem \ref{thm int-sym} (a) and Theorem \ref{thm rec bilocal} (a), respectively.  \end{proof}

There is one more interesting relation for the $4$-function
symbol, whose is a direct consequence of the explicit formula of the symbol.
\begin{theorem}
\label{thm symmetry}
Let $$R_{ijkl}=\{f_i,f_j,f_k,f_l\}_{C,P}.$$ Then $R_{ijkl}$ has
the same symmetry as the symmetry of a Riemann curvature tensor with respect to
permutations of the indices, namely
$$R_{ijkl}=-R_{jikl}=-R_{ijlk}=-R_{klij}.$$
\end{theorem}

\section{Second type of reciprocity laws}

\subsection{Bi-local symbols revisited}
In this Subsection, we define bi-local symbols, designed
for proofs of the second type of reciprocity laws for local symbols.  These bi-local symbols also satisfy reciprocity laws. 
Using them, in the following two sections, we establish the second type of reciprocity laws for the Parshin symbol and for a new $4$-function new symbol. By a second type of reciprocity law, we mean that the product of the local symbols is taken over all curves $C$ on the surface $X$, passing through a fixed point $P$. 

Let
$C_1,\dots,C_n$ be curves in $X$ intersecting at a point $P$.
Assume that $C_1,\dots,C_n$ are among the divisors of the rational
functions $f_1,\dots,f_4$. Let $\tilde{X}$ be the blow-up of $X$ at the point $P$. 
Assume that after the blow-up 
the curves above 
$C_1,\dots,C_n$ meet transversally the exceptional curve $E$ and no two of them
intersect at a point on the exceptional curve $E$.

Let $D$ be a curve on $\tilde{X}$ such that $D$
intersects $E$ in one point.
Setting
$$\tilde{P}_{k}=E\cap \tilde{C}_k,$$
where $\tilde{C}_k$ is the curve above $C_k$ after the blow-up, 
and
$$Q=E\cap D,$$
\begin{definition}
We define the following bi-local symbols
$$^{i,j}[f_1,\dots,f_{i+j}]_{C_k,P}^{(2),D}:=^{i,j}[f_1,\dots,f_{i+j}]_{E,\tilde{P}_{k}}^{(1),Q}.$$
\end{definition}

\begin{theorem}
\label{thm second bilocal rec} The following reciprocity laws for bi-local symbols hold:

\vspace{.3cm}
(a) $$\prod_{C_k} {^{1,2}}[f_1,f_2,f_3]_{C_k,P}^{(2),D}=1,$$

(b) $$\prod_{C_k} {^{2,1}}[f_1,f_2,f_3]_{C_k,P}^{(2),D}=1,$$

(c) $$\prod_{C_k} {^{2,2}}[f_1,f_2,f_3,f_4]_{C_k,P}^{(2),D}=1,$$
where the product is over the curves $C$, among the divisors of at least one of
the rational functions $f_1,\dots,f_4$, which pass through the point $P$.
\end{theorem}
The proof is reformulation of Theorem \ref{thm rec bilocal},
where the triple $(C_k,P,D)$ in the above Theorem correspond to the triple $(P,Q,C)$
with $P=C_k\cap E$ and $Q=D\cap E$ in Theorem \ref{thm rec bilocal}, where the curve $C$ in Theorem \ref{thm rec bilocal} corresponds to the curve $E$.

\subsection{Parshin symbol and its second reciprocity law.}
In this Subsection, we present an alternative refinement of the Parshin symbol in terms of bi-local symbols (Definition \ref{def refinement2}).  This implies the second reciprocity law for the Parshin symbol, since each of the bi-local symbols satisfy the second type of reciprocity laws (see Subsection 3.1).

\begin{definition}
\label{def refinement2}
We define the following bi-local symbol, useful for the proof of the second reciprocity law of the Parshin symbol
\begin{eqnarray*}
Pr_{C,E}^{D}=
\left({^{1,2}}[f_1,f_2,f_3]_{C,P}^{(2),D}\right)
\left({^{1,2}}[f_2,f_3,f_1]_{C,P}^{(2),D}\right)
\left({^{1,2}}[f_3,f_1,f_2]_{C,P}^{(2),D}\right)
\times\\
\times
\left({^{2,1}}[f_1,f_2,f_3]_{C,P}^{(2),D}\right)
\left({^{2,1}}[f_2,f_3,f_1]_{C,P}^{(2),D}\right)
\left({^{2,1}}[f_3,f_1,f_2]_{C,P}^{(2),D}\right)
,
\end{eqnarray*}
\end{definition}
Let $\tilde{P}=\tilde{C}\cap E$,  $Q=D\cap E$. Then
$$Pr_{C,E}^{D}=Pr_{E,\tilde{P}}^Q.$$
Similarly to the proof of Theorem \ref{thm parshin1}, we can remove the dependence of 
the bi-local symbol $Pr_{E,\tilde{P}}^Q$ on the base point $Q$.
\begin{definition}
The second Parshin symbol
$\{f_1,f_2,f_3\}^{(2)}_{C,E}$
is the symbol,
explicitly given by
\[\{f_1,f_2,f_3\}^{(2)}_{C,E}=
(-1)^K\left(f_1^{D_1}f_2^{D_2}f_3^{D_3}\right)(\tilde{P}),\]
where
$$D_1=
\left|
\begin{tabular}{ll}
$c_2$&$c_3$\\
$d_2$&$d_3$
\end{tabular}
\right|,
\mbox{ }
D_2=
\left|
\begin{tabular}{ll}
$c_3$&$c_1$\\
$d_3$&$d_1$
\end{tabular}
\right|,
\mbox{ }
D_3=
\left|
\begin{tabular}{ll}
$c_1$&$c_2$\\
$d_1$&$d_2$
\end{tabular}
\right|
$$
and
$$K=c_1c_2d_3+c_2c_3d_1+c_3c_1d_2+d_1d_2c_3+d_2d_3c_1+d_3d_1c_2,$$
with
$c_k=\ord_E(f_k)$ and $d_i=\ord_{\tilde{P}}((y^{-c_k}f_k)|_E).$
Here $y$ is a rational function rerpesenting an uniformizer at $E$ such that the components of the divisor of $y$ do not intersect at the point $\tilde{P}$.
\end{definition}

\begin{proposition}
\label{prop parshin2}
The second Parshin symbol is equal to the inverse of the Parshin symbol. More precisely,
\[\{f_1,f_2,f_3\}_{C,E}^{(2)}=(\{f_1,f_2,f_3\}_{C,P})^{-1}\]
\end{proposition}

Let 
\[a_i=\ord_C(f_i)\]
and 
\[b_i=\ord_P((x^{-a_i}f_i)|_C),\]
where $x$ is a rational function representing a uniformizer at $C$, whose support does not contain other components passing through the point $P$. 
\begin{lemma}
\label{lemma E}
With the above notation, the following holds 
\[\ord_E(f_i)=c_i=a_i+b_i.\]
\end{lemma}
\proof We still assume that after the blow-up the union of the support of the rational functions $f_1,f_2,f_3$ have normal crossings and no three  curves intersect at a point. Before the blow-up, let $C_1,\dots,C_n$ be all the components of the union of the support of the three rational functions that meet at  the point $P$. And let $E$ be the exceptional curve above the point $P$.
Then for $C=C_1$, we have \[b_i=\sum_{j=2}^n\ord_{C_j}(f_i)\] and
\[\ord_E(f_i)=\sum_{j=1}^n\ord_{C_j}(f_i).\]
That proves the Lemma.\qed

\proof (of Proposition \ref{prop parshin2})
Consider the pairs $(C,P)$ on the surface $X$ and $(E,\tilde{C})$ on on the blow-up $\tilde{X}$. 
Then by the above Lemma, we have
\begin{equation}
\label{eq change of var}
\left[
\begin{tabular}{ll}
$c_i$\\
$d_i$\\
\end{tabular}
\right]
=
\left[\begin{tabular}{ll}
$\ord_E(f_i)$\\
\\
$\ord_{\tilde{C}}(f_i)$
\end{tabular}
\right]
=
\left[
\begin{tabular}{ll}
1&1\\
1&0\\
\end{tabular}
\right]
\cdot
\left[
\begin{tabular}{ll}
$a_i$\\
\\
$b_i$
\end{tabular}
\right]
\end{equation}
The Parshin symbol is invariant under change of variables given by
$
\left[\begin{tabular}{ll}
1&0\\
1&1\\
\end{tabular}
\right].
$
Also the Parshin symbol is send to its reciprocal when we change the variables by a matrix 
$
\left[\begin{tabular}{ll}
0&1\\
1&0\\
\end{tabular}
\right].
$
That proves the Proposition. \qed

\begin{theorem}
\label{thm parshin2} (Second reciprocity law for the Parshin symbol)
We have
$$\prod_C \{f_1,f_2,f_3\}_{C,P}=1,$$
where the product is over the curves $C$ from the support of the divisors of the rational functions 
$\bigcup_{i=1}^3|div(f_i)|$, which pass through the point $P$. (For all other choices of curves $C$, the Parshin symbol will be equal to $1$.) Here we assume that the union of the support of the divisors 
$\bigcup_{i=1}^3|div(f_i)|$ in $\tilde{X}$ have normal crossings 
and no two components have a common point  with the exceptional curve $E$ in $\tilde{X}$ 
above the point $P$. 
We denote by $\tilde{X}$ the blow-up of $X$ at the point $P$.
\end{theorem}
\proof We can use Proposition \ref{prop parshin2} and the first reciprocity law for the Parshin symbol given in Theorem \ref{thm parshin1}. Then Theorem \ref{thm parshin2} follows. \qed

\subsection{The second $4$-function local symbol and its second reciprocity law}
In this Subsection, We define a second type of $4$-function local symbol (Definition \ref{def 4f explicit}), which satisfies the second type reciprocity laws. By a second reciprocity law, we mean that the product of the local symbols is taken over all curves $C$ on the surface $X$, which pass through a fixed point $P$. The $4$-function local symbol has a refinement (see Definition \ref{def 4f implicit}, which provides a proof of the second reciprocity law (Theorem \ref{thm rec2 4fn}).

\begin{definition}
\label{def 4f implicit}
We define a bi-local symbol, useful for the second reciprocity law for a new $4$-function local symbol. Let
\begin{eqnarray*}
PR_{C,P}^{D}
=
\left({^{2,2}}[f_1,f_2,f_3,f_4]_{C,E}^{(2),D}\right)\left({^{2,2}}[f_1,f_2,f_4,f_3]_{C,E}^{(2),D}\right)^{-1}
\times
\\
\times
\left({^{2,2}}[f_2,f_1,f_3,f_4]_{C,E}^{(2),D}\right)^{-1}\left({^{2,2}}[f_2,f_1,f_4,f_3]_{C,E}^{(2),D}\right).
\end{eqnarray*}

\end{definition}

Let \[L=(c_1d_2-c_2d_1)(c_3d_4-c_4d_3),\]
where
\[c_i=\ord_E(f_i),\]
\[d_i=\ord_{\tilde{P}}((x^{-a_i}f_i)|_E),\]
for a rational function $x$, representing a uniformizer at $E$, whose support does not contain other components passing through the point $\tilde{P}=E\cap \tilde{C}$. 
\begin{lemma}
\label{lemma 4f bilocal}
\[PR_{C,E}^{D}
=
(-1)^L
\left(\frac{\left(\frac{f_1^{c_2}}{f_2^{c_1}}\right)^{c_3d_4-c_4d_3}}
{\left(\frac{f_3^{c_4}}{f_4^{c_3}}\right)^{c_1d_2-c_2d_1}}(\tilde{P})\right)^{-1}
\frac{\left(\frac{f_1^{c_2}}{f_2^{c_1}}\right)^{c_3d_4-c_4d_3}}
{\left(\frac{f_3^{c_4}}{f_4^{c_3}}\right)^{c_1d_2-c_2d_1}}(Q),
\]
where $Q=D\cap E$.
\end{lemma}


It follows directly from Equation \ref{eq 4bilocal}  and Lemma \ref{lemma E}.

\begin{definition}
\label{def 4f explicit}
The second $4$-function local symbol has the following explicit representation:
\[\{f_1,f_2,f_3,f_4\}^{(2)}_{C,P}
=
(-1)^L
\left(\frac{\left(\frac{f_1^{a_2+b_2}}{f_2^{a_1+b_1}}\right)^{a_3b_4-b_3a_4}}
{\left(\frac{f_3^{a_4+b_4}}{f_4^{a_3+b_3}}\right)^{a_1b_2-b_1a_2}}(P)\right)^{-1}.
\]
\end{definition}

\begin{theorem}
\label{thm rec2 4fn} (Reciprocity law for the second $4$-function local symbol)
We have the following reciprocity law
$$\prod_C \{f_1,f_2,f_3,f_4\}^{(2)}_{C,P}=1,$$ 
where the product is over the curves $C$ from the support of the divisors of the rational functions 
$\bigcup_{i=1}^4|div(f_i)|$, 
which pass through the point $P$. Here we assume that the union of the support of the divisors 
$\bigcup_{i=1}^4|div(f_i)|$ in $\tilde{X}$ have normal crossings 
and no two components have a common point  with the exceptional curve $E$ in $\tilde{X}$ 
above the point $P$. 
We denote by $\tilde{X}$ the blow-up of $X$ at the point $P$.
\end{theorem}
\proof 
Using Theorem \ref{thm second bilocal rec}, we obtain a reciprocity law for the bi-local symbol $PR_{C,E}^{(2),D}$. Multiplying each symbol by the same constant, depending only on $Q$, we can remove the dependence on $Q$. Explicitly, the separation between the dependence on $D$ and the second $4$ function local symbol are given in Lemma \ref{lemma 4f bilocal}.  Then we can use Lemma \ref{lemma E} in order to express 
the coefficients $c_i$ and $d_i$ in terms of $a_i$ and $b_i$, which implies the reciprocity law stated in the Theorem \ref{thm rec2 4fn}.
\qed.


\section{An alternative proof of the reciprocity laws for the $4$-function local symbols}
In this Section, we give alternative proofs of the two reciprocity laws of the $4$-function local symbol, based in Milnor $K$-theory. We will use the $K$-theoretic interpretation of the $4$-function local symbol, presented in the Appendix by M. Kerr.
\begin{definition}
\label{def sign}
The quotient of the $K$-theoretic symbol $^K[f_1,f_2,f_3,f_4]_{C,P}^{(1)}$,  from the Appendix, and the $4$-function local symbol $\{f_1,f_2,f_3,f_4\}^{(1)}_{C,P}$ from Definition  \ref{def 4-function} 
is given by
\[(f_1,f_2,f_3,f_4)_{C,P}^{(1)}=(-1)^{a_1a_2a_3b_4+a_2a_3a_4b_1+a_3a_4a_1b_2+a_4a_1a_2b_3}.\]
Here 
$a_k=\ord_C(f_k)$ and $b_k=\ord_P((x^{-a_k}f_k)|_C),$
where $x$ is a rational function representing an uniformizer at $C$ such that $P$ is not an intersection point of the irreducible components of the support of the divisor $(x)$.
\end{definition}
For each point $P$ on a fixed curve $C$ the values $a_k$ remain the same. Therefore, we have the following interpretation in terms of integrals. Let \[\omega_k=(-a_k\frac{dx}{x}+\frac{df_k}{f_k})|C\] be a differential form on the curve $C$. Then
\[b_k(P)=\frac{1}{2\pi i}Res_P(\omega_k)\]
\begin{proposition} We can express the sign $(f_1,f_2,f_3,f_4)_{C,P}$ in terms of residues

\vspace{.3cm}
\begin{tabular}{lll}
$(f_1,f_2,f_3,f_4)^{(1)}_{C,P}
=$&$\exp\left(\frac{1}{2}\left(a_1a_2a_3Res_P(\omega_4)+
a_2a_3a_4Res_P(\omega_1)+\right.\right.$\\
\\
&$\left.\left.+a_3a_4a_1Res_P(\omega_2)+
a_4a_1a_2Res_P(\omega_3)\right)\right)$
\end{tabular}
\end{proposition}
\begin{theorem}
The sign $(f_1,f_2,f_3,f_4)^{(1)}_{C,P}$ is also a symbol, satisfying the following reciprocity law:
\[\prod_{P}(f_1,f_2,f_3,f_4)^{(1)}_{C,P}=1,\]
where the product is over all points $P$ of the curve $C$.
\end{theorem}

\proof It follows from the fact that the sum of the residues on a curve is equal to zero and from the previous Proposition. \qed

\begin{theorem}
\label{thm rec1 Ksym}
The $K$-theoretic symbol satisfies the following reciprocity law
$$\prod_C \mbox{ }^K[f_1,f_2,f_3,f_4]^{(1)}_{C,P}=1.$$
where the product is over all points $P$ of the curve $C$.
\end{theorem}

The proof follows directly from the $K$-theoretic definition given in the Appendix.

\proof (an alternative proof of Theorem \ref{thm rec1 4fn})
Using the reciprocity law for the $K$-theoretic symbol $^K[f_1,f_2,f_3,f_4]^{(1)}_{C,P}$ such as in the Appendix and  the above Theorem,
we obtain another proof of the reciprocity law for the $4$-function local symbol. \qed

Now, we proceed toward an alternative  proof of the second type of reciprocity laws for the new 
$4$-function local symbol.

Let $E$ be the exceptional curve for the blowup of $X$ at the point $P$. Let $\tilde{C}$ be the irreducible component sitting above the curve $C$ in the blow-up. We define $\tilde{P}=\tilde{C}\cap E.$
A direct observation leads to
\[\{f_1,f_2,f_3,f_4\}^{(2)}_{C,P}=\left(\{f_1,f_2,f_3,f_4\}^{(1)}_{E,\tilde{P}}\right)^{-1}\]
for the $4$-function local symbols.
Similarly we define
\begin{equation}
\label{eq sign2}
(f_1,f_2,f_3,f_4)^{(2)}_{C,P}=\left((f_1,f_2,f_3,f_4)^{(1)}_{E,\tilde{P}}\right)^{-1}
\end{equation}
for the sign
and
\begin{equation}
\label{eq Ksym2}
^K[f_1,f_2,f_3,f_4]^{(2)}_{C,P}=  \left(\mbox{ }^K[f_1,f_2,f_3,f_4]^{(1)}_{E,\tilde{P}}\right)^{-1}
\end{equation}
for the $K$-theoretic symbol.

\begin{theorem} 
\label{thm rec2 sign and K}
For the sign and the $K$ theoretic symbol we have a second type of reciprocity laws.
\[\prod_C(f_1,f_2,f_3,f_4)^{(2)}_{C,P}=1\]
and
\[\prod_C\mbox{ }^K[f_1,f_2,f_3,f_4]^{(2)}_{C,P}=1,\]
where the product is taken over all curves $C$, passing through the point $P$. Here we assume that the union of the support of the divisors 
$\bigcup_{i=1}^4|div(f_i)|$ in $\tilde{X}$ have normal crossings 
and no two components have a common point  with the exceptional curve $E$ in $\tilde{X}$ 
above the point $P$. 
We denote by $\tilde{X}$ the blow-up of $X$ at the point $P$.
\end{theorem}
\proof
For the $K$-theoretic symbol we have
\[\prod_C\mbox{ }^K[f_1,f_2,f_3,f_4]^{(2)}_{C,P}= 
\left(\prod_{\tilde{P}} \mbox{ }^K[f_1,f_2,f_3,f_4]^{(1)}_{E,\tilde{P}}\right)^{-1}=1.\]
The first equality follows from the definition of 
$^K[f_1,f_2,f_3,f_4]^{(2)}_{C,P}$ and the second equality from Theorem \ref{thm rec1 Ksym}. \qed

\proof (an alternative proof of Theorem \ref{thm rec2 4fn}) We have the following equalities
\[\prod_C\{f_1,f_2,f_3,f_4\}^{(2)}_{C,P}
=\prod_C\mbox{ }^K[f_1,f_2,f_3,f_4]^{(2)}_{C,P}
\prod_C\mbox{ }^K[f_1,f_2,f_3,f_4]^{(2)}_{C,P}
=1.\]
The first equality follows from Definition \ref{def sign} and Equations \eqref{eq sign2} and \eqref{eq Ksym2}. The second equality follows from Theorem \ref{thm rec2 sign and K}. \qed

\par \vspace{\baselineskip}

\subsection*{Appendix A. By Matt Kerr}



There is a well-known K-theoretic approach to the Parshin symbol, which 
we shall recall below.  The purpose of this appendix is to provide (up to 
sign) a K-theoretic interpretation for the 4-function symbol.

To begin, note that if only two components $C$ and $C'$ of
$\bigcup_i|(f_i)|$ meet at a point $P\in X$, then the Parshin
symbol has the {\it{local symmetry property}}
$$\{f_1,f_2,f_3\}_{C,P}=(\{f_1,f_2,f_3\}_{C',P})^{-1}$$
as does the (second) $4$-function symbol. This is just a special case of
the second reciprocity law.

Now it is well known that the Parshin symbol may be computed by
the composition
\begin{equation*}
\xymatrix{ K^M_3(\C(X)) \ar @/^1pc/ [rr]^{{\cal{P}}_{C,P}}  \ar [r]_{\mspace{10mu}Tame_C} & K_2(\C(C)) \ar [r]_{\mspace{30mu}Tame_P} & \C^\times.}
\end{equation*}
Since ${\cal{P}}_{C,P}$ is invariant under blow-up and satisfies
the local symmetry property, this reduces checking the two
reciprocity laws to Weil reciprocity.

One is tempted to believe that the $4$-function symbol
$\{f_1,f_2,f_3,f_4\}^{(1)}_{C,P}$ (Definition 2.13) can be identified
with the image of $\{f_1,f_2\}\otimes\{f_3,f_4\}$ under the
composition
\begin{equation*}
\xymatrix{K_2(\C(X))^{\otimes 2} \ar @/^1pc/ [rrr]^{\mathcal{Q}_{C,P}}  \ar [r]_{Tame_C^{\otimes 2}} & (\C(C)^{\times})^{\otimes 2} \ar @{->>} [r]  &  K_2(\C(C)) \ar [r]_{\mspace{30mu} Tame_P} & \C^\times }
\end{equation*}
and to argue in the same manner. Indeed, the first reciprocity law
for $\mathcal{Q}_{C,P}$ again follows from Weil reciprocity, and
it is also invariant under blow-up.

However, a short computation shows that $\mathcal{Q}_{C,P}$ and
Definition \ref{def 4-function} differ by the factor
$$(-1)^{a_1b_2b_3b_4+b_1a_2b_3b_4+b_1b_2a_3b_4+b_1b_2b_3a_4}.$$
\noindent Indeed, we have
$$^K[f_1,f_2,f_3,f_4]_{C,P}^{(1)}:=$$
$$ Tame_P\{Tame_C\{f_1,f_2\},Tame_C\{f_3,f_4\}\}=$$
$$Tame_P\{Tame_C\{x^{a_1}y^{b_1}g_1,x^{a_2}y^{b_2}g_2\},Tame_C\{x^{a_3}y^{b_3}g_3,x^{a_4}y^{b_4}g_4\}\}
=$$
$$Tame_P\left\{(-1)^{a_1a_2}x^{a_1b_2-a_2b_1}\frac{(g_1|_C)^{a_2}}{(g_2|_C)^{a_1}},
(-1)^{a_3a_4}x^{a_3b_4-a_4b_3}\frac{(g_3|_C)^{a_4}}{(g_4|_C)^{a_3}}\right\}
=$$
$$ (-1)^{(a_1b_2-a_2b_1)(a_3b_4-a_4b_3)}
\frac{\left\{(-1)^{a_1a_2}\left(\frac{(g_1|_P)^{a_2}}{(g_2|_P)^{a_1}}\right)\right\}^{(a_3b_4-a_4b_3)}}
{\left\{(-1)^{a_3a_4}\left(\frac{(g_3|_P)^{a_4}}{(g_4|_P)^{a_3}}\right)\right\}^{(a_1b_2-a_2b_1)}}
=$$
$$ (-1)^{b_1a_2a_3a_4+a_1b_2a_3a_4+a_1a_2b_3a_4+a_1a_2a_3b_4}
\{f_1,f_2,f_3,f_4\}_{C,P}^{(1)}.$$












\paragraph*{Acknowledgments:}
M. Kerr thanks I. Horozov for discussions and the NSF for support under grant DMS-1068974.

\renewcommand{\em}{\textrm}

\begin{small}

\renewcommand{\refname}{ {\flushleft\normalsize\bf{References}} }
    
\end{small}
Ivan E. Horozov\\
\begin{small}
Washington University in St Louis\\
Department of Mathematics\\
Campus Box 1146\\
One Brookings Drive\\
St Louis, MO 63130\\
USA\\
\end{small}
\\
Matt Kerr\\
\begin{small}
Washington University in St Louis\\
Department of Mathematics\\
Campus Box 1146\\
One Brookings Drive\\
St Louis, MO 63130\\
USA
\end{small}

\end{document}